\documentclass[leqno,12pt]{article} 
\usepackage{amsmath, amssymb}
\usepackage{amsthm} 
\newcommand\supp{\mathop{\rm supp}}
\newcommand\esssup{\mathop{\rm ess \, sup}}
\newcommand\essinf{\mathop{\rm ess \, inf}}

\theoremstyle{plain} 
\newtheorem{theorem}{\indent\sc Theorem}[section]
\newtheorem{lemma}[theorem]{\indent\sc Lemma}
\newtheorem{corollary}[theorem]{\indent\sc Corollary}
\newtheorem{proposition}[theorem]{\indent\sc Proposition}

\theoremstyle{definition} 
\newtheorem{definition}[theorem]{\indent\sc Definition}
\newtheorem{remark}[theorem]{\indent\sc Remark}

\makeatletter
\def\address#1#2{\begingroup
\noindent\parbox[t]{7.8cm}{%
\small{\scshape\ignorespaces#1}\par\vskip1ex
\noindent\small{\itshape E-mail address}%
\/: #2\par\vskip4ex}\hfill%
\endgroup}%
\makeatother
\title{\uppercase{Inequalities for weighted spaces with variable exponents}} 
\author{
\textsc{Pablo Rocha} 
}
\date{} 
\begin{document}

\maketitle


\begin{abstract}
In this article we obtain an "off-diagonal" version of the Fefferman-Stein vector-valued maximal inequality on weighted Lebesgue spaces with variable exponents. As an application of this result and the atomic decomposition developed in \cite{Ho1} we prove, for certain exponents 
$q(\cdot)$ in $\mathcal{P}^{\log}(\mathbb{R}^{n})$ and certain weights $\omega$, that the Riesz potential $I_{\alpha}$, with
$0 < \alpha < n$, can be extended to a bounded operator from $H^{p(\cdot)}_{\omega}(\mathbb{R}^{n})$ into 
$L^{q(\cdot)}_{\omega}(\mathbb{R}^{n})$, for $\frac{1}{p(\cdot)} := \frac{1}{q(\cdot)} + \frac{\alpha}{n}$.
\end{abstract}

\section{Introduction}

The theory of variable exponent Hardy spaces $H^{p(\cdot)}(\mathbb{R}^{n})$ was developed independently by Nakai and Sawano in \cite{Nakai} and by Cruz-Uribe and Wang in \cite{Uribe3}. Both theories prove equivalent definitions in terms of maximal operators using different approaches. In \cite{Nakai, Uribe3}, one of their main goals is the atomic decomposition of elements in $H^{p(\cdot)}$, as an application of the atomic decomposition they proved that singular integrals are bounded on $H^{p(\cdot)}$. Later in \cite{Rocha1}, the author jointly with Urciuolo proved  the $H^{p(\cdot)} - L^{q(\cdot)}$ boundedness of certain generalized Riesz potentials and the $H^{p(\cdot)} - H^{q(\cdot)}$ boundedness of Riesz potential via the infinite atomic and molecular decomposition developed in \cite{Nakai}. In \cite{Rocha3}, the author gave another proof of the results obtained in \cite{Rocha1}, but by using the finite atomic decomposition developed in \cite{Uribe3}.

Recently, Kwok-Pun Ho in \cite{Ho1} developed the weighted theory for variable Hardy spaces on $\mathbb{R}^{n}$. He established the atomic decompositions for the weighted Hardy spaces with variable exponents $H^{p(\cdot)}_{\omega}(\mathbb{R}^{n})$ and also revealed some intrinsic structures of atomic decomposition for Hardy type spaces. His results generalize the infinite atomic decompositions obtained in \cite{Garcia, Bui, Torch, Nakai}.

The main results of this work are contained in Theorems \ref{Feff-Stein ineq} and \ref{weighted Hp-Lq} below. The first result concerns to 
an "off-diagonal" version of the Fefferman-Stein vector-valued maximal inequality on weighted variable Lebesgue spaces (originally due to Harboure, Mac\'ias and Segovia \cite{Harboure}). Our second result states, for certain exponent functions $q(\cdot)$ and certain weights 
$\omega$, the boundedness of Riesz potential $I_{\alpha}$ from $H^{p(\cdot)}_{\omega}$ into $L^{p(\cdot)}_{\omega}$ where $0 < \alpha < n$ and $\frac{1}{p(\cdot)} :=  \frac{1}{q(\cdot)} + \frac{\alpha}{n}$. This result is achieved via the off-diagonal Fefferman-Stein 
inequality (Theorem \ref{Feff-Stein ineq}), the atomic decomposition established in \cite{Ho1}, and the following additional assumption: for every cube $Q \subset \mathbb{R}^{n}$
\begin{equation} \label{hipotesis admisible}
\| \chi_Q \|_{L^{q(\cdot)}_{\omega}} \approx |Q|^{-\alpha/n} \| \chi_Q \|_{L^{p(\cdot)}_{\omega}}.
\end{equation}
If $\omega \equiv 1$ and $q(\cdot)$ is log-H\"older continuous (locally and at infinite), then (\ref{hipotesis admisible}) holds. This case was proved in \cite{Rocha1}. In this article we will give non trivial examples of power weights satisfying (\ref{hipotesis admisible}). So, (\ref{hipotesis admisible}) is an admissible hypothesis.

This paper is organized as follows. Section 2 gives the definition of weighted variable Lebesgue spaces $L^{p(\cdot)}_{\omega}$ 
and weighted variable Hardy spaces $H^{p(\cdot)}_{\omega}$ and some of their preliminary results. Section 3 presents some basics 
properties of the set of weights $\mathcal{W}_{p(\cdot)}$ used to define $H^{p(\cdot)}_{\omega}$ and some results about the size of the cubes in the $L^{p(\cdot)}_{\omega}$-norm. The off-diagonal version of the Fefferman-Stein inequality on weighted variable Lebesgue spaces is established in Section 4. The $H^{p(\cdot)}_{\omega} - L^{q(\cdot)}_{\omega}$ boundedness of Riesz potential $I_{\alpha}$ is proved in Section 5. Finally, Section 6 gives non trivial examples of power weights satisfying (\ref{hipotesis admisible}).

\

{\bf Notation.} The symbol $A\lesssim B$ stands for the inequality $A \leq c B$ for some constant $c$. The symbol $A \approx B$ 
stands for $B \lesssim A \lesssim B$. We denote by $Q\left( z,r\right)$ the cube centered at $z \in \mathbb{R}^{n}$ with side lenght $r$. Given a cube $Q = Q(z, r)$, we set $kQ = Q(z,kr)$ and $l(Q) = r$. For a measurable subset $E \subseteq \mathbb{R}^{n}$ we denote 
by $|E|$ and $\chi_{E}$ the Lebesgue measure of $E$ and the characteristic function of $E$ respectively. Given a real number $s \geq 0$, 
we write $\lfloor s \rfloor$ for the integer part of $s$. As usual we denote with $\mathcal{S}(\mathbb{R}^{n})$ the space of smooth and rapidly decreasing functions and with $\mathcal{S}'(\mathbb{R}^{n})$ the dual space. If $\beta$ is the multiindex 
$\beta =(\beta_{1},...,\beta _{n})$, then $|\beta| =\beta _{1}+...+\beta _{n}.$
 
\section{Preliminaries}

Let $p(\cdot) : \mathbb{R}^{n} \to (0, \infty)$ be a measurable function. Given a measurable set $E$, let
\[
p_{-}(E) = \essinf_{ x \in E } p(x), \,\,\,\, \text{and} \,\,\,\, p_{+}(E) = \esssup_{x \in E} p(x).
\]
When $E = \mathbb{R}^{n}$, we will simply write $p_{-} := p_{-}(\mathbb{R}^{n})$ and $p_{+} := p_{+}(\mathbb{R}^{n})$.

Given a measurable function $f$ on $\mathbb{R}^{n}$, define the modular $\rho$ associated with $p(\cdot)$ by
\[
\rho(f) = \int_{\mathbb{R}^{n}} |f(x)|^{p(x)} dx.
\]
We define the variable Lebesgue space $L^{p(\cdot)} = L^{p(\cdot)}(\mathbb{R}^{n})$ to be the set of all measurable functions $f$ such that, for some $\lambda > 0$, $\rho \left( f/\lambda \right) < \infty$. This becomes a quasi normed space when equipped with the Luxemburg norm
\[
\| f \|_{L^{p(\cdot)}} = \inf \left\{ \lambda > 0 : \rho \left( f/\lambda \right) \leq 1 \right\}.
\]

Given a weight $\omega$, i.e.: a locally integrable function on $\mathbb{R}^{n}$ such that $0 < \omega(x) < \infty$ almost everywhere, we define the weighted variable Lebesgue space $L^{p(\cdot)}_{\omega}$ as the set of all measurable functions 
$f : \mathbb{R}^{n} \to \mathbb{C}$ such that $\| f \omega \|_{L^{p(\cdot)}} < \infty$. If $f \in L^{p(\cdot)}_{\omega}$, we define 
its "norm" by
\[
\| f \|_{L^{p(\cdot)}_\omega} := \| f \omega \|_{L^{p(\cdot)}}. 
\]

\begin{lemma} \label{potencia s}
Given a measurable function $p(\cdot) : \mathbb{R}^{n} \to (0, \infty)$ with $0 < p_{-} \leq p_{+} < \infty$ and a weight 
$\omega$, then for every $s > 0$,
\[
\| f \|_{L^{p(\cdot)}_\omega}^{s} = \| |f|^{s} \|_{L^{p(\cdot)/s}_{\omega^{s}}}.
\]
\end{lemma}

\begin{proof}
The condition $p_{+} < \infty$ implies that $|\{ x : p(x) = \infty \}| = 0$. Then, for every $s>0$,
\[
\| |f|^{s} \|_{L^{p(\cdot)/s}_{\omega^{s}}} = 
\inf \left\{ \lambda > 0 : \int \left( \frac{|f(x)| \, \omega(x)}{\lambda^{1/s}} \right)^{p(x)} dx \leq 1 \right\}
\]
\[
= \inf \left\{ \mu^{s} > 0 : \int \left( \frac{|f(x)| \, \omega(x)}{\mu} \right)^{p(x)} dx \leq 1 \right\} = 
\| f \|_{L^{p(\cdot)}_\omega}^{s}.
\]
\end{proof}

For a measurable function $p(\cdot) : \mathbb{R}^{n} \to [1, \infty)$, its conjugate function $p'(\cdot)$ is defined by 
$\frac{1}{p(x)} + \frac{1}{p'(x)} = 1$. We have the following generalization of H\"older's
inequality and an equivalent expression for the $L^{p(\cdot)}_{\omega}$-norm.

\begin{lemma} (H\"older's inequality) \label{Holder ineq}
Let $p(\cdot) : \mathbb{R}^{n} \to [1, \infty)$ be a measurable function and $\omega$ be a locally integrable function such that 
$0 < \omega(x) < \infty$ almost everywhere. Then, there exists a constant $C > 0$ such that
\[
\int_{\mathbb{R}^{n}} |f(x) g(x)| dx \leq C \| f \|_{L^{p(\cdot)}_{\omega}} \| g \|_{L^{p'(\cdot)}_{\omega^{-1}}}.
\]
\end{lemma}

\begin{proof} The lemma follows from \cite[Lemma 3.2.20]{Diening2}.
\end{proof}

\begin{proposition} \label{norma equivalente}
Let $p(\cdot) : \mathbb{R}^{n} \to [1, \infty)$ be a measurable function and $\omega$ be a locally integrable function such that 
$0 < \omega(x) < \infty$ almost everywhere. Then
\[
\| f \|_{L^{p(\cdot)}_{\omega}} \approx \sup \left\{ \int_{\mathbb{R}^{n}} |f(x) g(x)| dx : \| g \|_{L^{p'(\cdot)}_{\omega^{-1}}} \leq 1
\right\}.
\]
\end{proposition}

\begin{proof} The proposition follows from \cite[Corollary 3.2.14]{Diening2}.
\end{proof}

Now, we briefly present the basics of weighted variable Hardy spaces theory and recall the atomic decomposition
developed by K.-P. Ho in \cite{Ho1}.

We introduce the weights used in \cite{Ho1} to define weighted Hardy spaces with variable exponents.

\begin{definition} \label{pesos Wp}
Let $p(\cdot) : \mathbb{R}^{n} \to (0, \infty)$ be a measurable function with $0 < p_{-} \leq p_{+} < \infty$. We define
$\mathcal{W}_{p(\cdot)}$ as the set of all weights $\omega$ such that
\[
(i) \,\, \text{there exists} \,\, 0 < p_{\ast} < \min \{1, p_{-} \} \,\, \text{such that} \,\, 
\| \chi_Q \|_{L^{p(\cdot)/p_{\ast}}_{\omega^{p_{\ast}}}} < \infty, \,\,\, \text{and}
\]
\[ 
\| \chi_Q \|_{L^{(p(\cdot)/p_{\ast})'}_{\omega^{-p_{\ast}}}} < \infty, \,\,\, \text{for all cube} \,\, Q;
\]
\[
(ii) \,\, \text{there exist $\kappa > 1$ and $s > 1$ such that Hardy-Littlewood maximal}
\]
\[ 
\text{operator $M$ is bounded on $L^{(sp(\cdot))'/\kappa}_{\omega^{-\kappa/s}}$}. 
\]
\end{definition}

\begin{remark}
In \cite[Definition 2.3]{Ho1}, the author considers $p_{\ast} = \min \{ 1, p_{-} \}$. We observe that the whole theory is still valid if we take $0 < p_{\ast} < \min \{1, p_{-} \}$ instead of $p_{\ast} = \min \{ 1, p_{-} \}$. This slight modification allows us to compare the sets
$\mathcal{W}_{q(\cdot)}$ and $\mathcal{W}_{p(\cdot)}$ related by $\frac{1}{p(\cdot)} - \frac{1}{q(\cdot)} = \frac{\alpha}{n}$ with
$0 < \alpha < n$ (see Proposition \ref{WqcWp} below).
\end{remark}

For a measurable function $p(\cdot) : \mathbb{R}^{n} \to (0, \infty)$ such that $0 < p_{-}\leq p_{+} < \infty$ and 
$\omega \in \mathcal{W}_{p(\cdot)}$, in \cite{Ho1} the author give a variety of distinct approaches, based on
differing definitions, all lead to the same notion of weighted variable Hardy space $H^{p(\cdot)}_{\omega}$. 

We recall some terminologies and notations from the study of maximal functions. Given $N \in \mathbb{N}$, define 
\[
\mathcal{F}_{N}=\left\{ \varphi \in \mathcal{S}(\mathbb{R}^{n}):\sum\limits_{\left\vert \mathbf{\beta }\right\vert \leq N}\sup\limits_{x\in \mathbb{R}%
^{n}}\left( 1+\left\vert x\right\vert \right) ^{N}\left\vert \partial^{%
\mathbf{\beta }}\varphi (x)\right\vert \leq 1\right\}.
\]
For any $f \in \mathcal{S}'(\mathbb{R}^{n})$, the grand maximal function of $f$ is given by 
\[
\mathcal{M}f(x)=\sup\limits_{t>0}\sup\limits_{\varphi \in \mathcal{F}_{N}}\left\vert \left( \varphi_t \ast f\right)(x) \right\vert,
\]
where $\varphi_t(x) = t^{-n} \varphi(t^{-1} x)$. 

\begin{definition} Let $p(\cdot):\mathbb{R}^{n} \to ( 0,\infty)$, $0 < p_{-} \leq p_{+} < \infty $, and 
$\omega \in \mathcal{W}_{p(\cdot)}$. The weighted variable Hardy space $H^{p(.)}_{\omega}(\mathbb{R}^{n})$ is the set of 
all $f \in \mathcal{S}'(\mathbb{R}^{n})$ for which $\| \mathcal{M}f \|_{L^{p(\cdot)}_{\omega}} < \infty$. In this case we define 
$\| f \|_{H^{p(\cdot)}_{\omega}} := \| \mathcal{M}f \|_{L^{p(\cdot)}_{\omega}}$.
\end{definition}

\begin{definition} Let $p(\cdot):\mathbb{R}^{n} \to ( 0,\infty)$, $0 < p_{-} \leq p_{+} < \infty $, $p_{0} > 1$, and 
$\omega \in \mathcal{W}_{p(\cdot)}$. Fix an integer $d \geq 1$. A function $a(\cdot)$ on $\mathbb{R}^{n}$ is called a 
$\omega-(p(\cdot), p_{0}, d)$ atom if there exists a cube $Q$ such that

$a_{1})$ $\supp ( a ) \subset Q$,

$a_{2})$ $\| a \|_{L^{p_{0}}}\leq \frac{| Q |^{\frac{1}{p_{0}}}}{\| \chi _{Q} \|_{L^{p(\cdot)}_{\omega}}}$,

$a_{3})$ $\displaystyle{\int} x^{\beta}  a(x) \, dx = 0$ for all $| \beta | \leq d$.
\end{definition}

Next, we introduce two indices, which are related to the intrinsic structure of the atomic decomposition of $H^{p(\cdot)}_{\omega}$. Given $\omega \in \mathcal{W}_{p(\cdot)}$, we write
\[
s_{\omega, \, p(\cdot)} := \inf \left\{ s \geq 1 : M \,\, \text{is bounded on} \,\, L^{(sp(\cdot))'}_{\omega^{-1/s}} \right\}
\]
and
\[
\mathbb{S}_{\omega, \, p(\cdot)} := \left\{ s \geq 1 : M \,\, \text{is bounded on} \,\, L^{(sp(\cdot))'/\kappa}_{\omega^{-\kappa/s}} \,\,
\text{for some} \,\, \kappa > 1 \right\}.
\]
Then, for every fixed $s \in \mathbb{S}_{\omega, \, p(\cdot)}$, we define
\[
\kappa_{\omega, \, p(\cdot)}^{s} := \sup \left\{ \kappa > 1 : M \,\, \text{is bounded on} \,\, L^{(sp(\cdot))'/\kappa}_{\omega^{-\kappa/s}} \right\}.
\]
The index $\kappa_{\omega, \, p(\cdot)}^{s}$ is used to measure the left-openness of the boundedness of $M$ on the family 
$\left\{ L^{(sp(\cdot))'/\kappa}_{\omega^{-\kappa/s}} \right\}_{\kappa > 1}$. The index $s_{\omega, \, p(\cdot)}$ is related to the vanishing moment condition and the index $\kappa_{\omega, \, p(\cdot)}^{s}$ is related to the size condition of the atoms 
(see \cite[Theorems 5.3 and 6.3]{Ho1}).

The following three results are crucial in obtaining the $H^{p(\cdot)}_{\omega} - L^{q(\cdot)}_{\omega}$ boundedness of the Riesz potential.
The first is a supporting result \cite[Lemma 5.4]{Ho1}, the second is a combination of \cite[Theorem 6.2]{Ho1} and 
\cite[Theorem 3.1]{Rocha2}, the last one was proved in \cite[Proposition 2.1]{Ho2}. 

\begin{proposition} (\cite[Lemma 5.4]{Ho1}) \label{bk functions}
Let $p(\cdot) : \mathbb{R}^{n} \to (0, \infty)$ be a measurable function with $0 < p_{-} \leq p_{+} < \infty$ and 
$\omega \in \mathcal{W}_{p(\cdot)}$. Let $s \in \mathbb{S}_{\omega, \, p(\cdot)}$ and $\{ \lambda_j \}_{j \in \mathbb{N}}$ be a sequence of scalars. If $r > (\kappa_{\omega, \, p(\cdot)}^{s})'$, $\{ b_j \}_{j \in \mathbb{N}}$ is a sequence in $L^{r}$ such that, for every $j$, $\supp(b_j)$ is contained in a cube $Q_{j}$ and
\[
\| b_j \|_{L^{r}} \leq A_j |Q_j|^{1/r},
\]
where the $A_j$'s are positive scalars for all $j \in \mathbb{N}$, then
\[
\left\| \sum_{j \in \mathbb{N}} \lambda_j b_j  \right\|_{L^{sp(\cdot)}_{\omega^{1/s}}} \leq C 
\left\| \sum_{j \in \mathbb{N}} A_j |\lambda_j| \chi_{Q_j}  \right\|_{L^{sp(\cdot)}_{\omega^{1/s}}}
\]
for some $C > 0$ independent of $\{ \lambda_j \}_{j \in \mathbb{N}}$, $\{ b_j \}_{j \in \mathbb{N}}$ and $\{ A_j \}_{j \in \mathbb{N}}$.
\end{proposition}

\begin{theorem} \label{w atomic decomp}
Let $1 < p_0 < \infty$, $p(\cdot) : \mathbb{R}^{n} \to (0, \infty)$ be a measurable function with $0 < p_{-} \leq p_{+} < \infty$ and 
$\omega \in \mathcal{W}_{p(\cdot)}$. Then, for every $s > 1$ fixed, and every 
$f \in H^{p(\cdot)}_{\omega}(\mathbb{R}^{n}) \cap L^{s}(\mathbb{R}^{n})$ there exist a sequence of scalars $\{ \lambda_j \}$, a sequence of cubes $\{ Q_j \}$ and $\omega - (p(\cdot), p_0, \lfloor n s_{\omega, \, p(\cdot)} - n \rfloor)$ atoms $a_j$ supported on $Q_j$ such that 
\begin{equation} \label{Hpw norm}
\left\| \sum_{j} \left( \frac{|\lambda_j |}{\| \chi_{Q_j} \|_{L^{p(\cdot)}_{\omega}}} \right)^{\theta} \chi_{Q_j} \right\|_{L^{p(\cdot)/\theta}_{\omega^{\theta}}}^{1/\theta} \leq C \| f \|_{H^{p(\cdot)}_{\omega}}, \,\,\, \text{for all} \,\,\, 0 < \theta < \infty,
\end{equation}
and $f = \sum_{j} \lambda_j a_j$ converges in $L^{s}(\mathbb{R}^{n})$.
\end{theorem}

\begin{proof}
The existence of a such atomic decomposition as well as the validity of inequality in (\ref{Hpw norm}) are guaranteed by 
\cite[Theorem 6.2]{Ho1}. The construction of a such atomic decomposition is the same that the given for the classical Hardy spaces 
(see \cite[Chapter III]{Stein}). So, following the proof in \cite[Theorem 3.1]{Rocha2}, we obtain the convergence in 
$L^{s}(\mathbb{R}^{n})$.
\end{proof}

We define the set $\mathcal{S}_{0}(\mathbb{R}^{n})$ by
\[
\mathcal{S}_{0}(\mathbb{R}^{n}) =\left\{ \varphi \in \mathcal{S}(\mathbb{R}^{n}) : \int x^{\beta} \varphi(x) dx = 0, \,\, \text{for all} \, 
\beta \in \mathbb{N}_{0}^{n} \right\}.
\]

We say that an exponent function $p(\cdot) : \mathbb{R}^{n} \to (0, \infty)$ such that $0 < p_{-} \leq p_{+} < \infty$ belongs 
to $\mathcal{P}^{log}(\mathbb{R}^{n})$, if there exist two positive constants $C$ and $C_{\infty}$ such that $p(\cdot)$ satisfies the 
local log-H\"older continuity condition, i.e.:
\[
|p(x) - p(y)| \leq \frac{C}{-\log(|x-y|)}, \,\,\, |x-y| \leq \frac{1}{2},
\]
and is log-H\"older continuous at infinity, i.e.:
\[
|p(x) - p_{\infty}| \leq \frac{C_{\infty}}{\log(e+|x|)}, \,\,\, x \in \mathbb{R}^{n},
\]
for some $p_{-} \leq p_{\infty} \leq p_{+}$. 

\begin{theorem} (\cite[Proposition 2.1]{Ho2}) \label{dense}
Let $p(\cdot) \in \mathcal{P}^{\log}(\mathbb{R}^{n})$ with $0 < p_{-} \leq p_{+} < \infty$. If $\omega \in \mathcal{W}_{p(\cdot)}$, then
$\mathcal{S}_{0}(\mathbb{R}^{n}) \subset H^{p(\cdot)}_{\omega}(\mathbb{R}^{n})$ densely.
\end{theorem}

\section{Auxiliary Results}

In this section we prove some basics properties of the set $\mathcal{W}_{p(\cdot)}$ and some results about the size of the cubes in the
$L^{p(\cdot)}_{\omega}$-norm.

\begin{proposition} \label{WqcWp}
Let $0 < \alpha < n$ and let $q(\cdot) : \mathbb{R}^{n} \to (0, \infty)$ be a measurable function such that $0 < q_{-} \leq q_{+} < \infty$. If $\omega \in \mathcal{W}_{q(\cdot)}$ and $\frac{1}{p(\cdot)} := \frac{1}{q(\cdot)} + \frac{\alpha}{n}$, then 
$\omega \in \mathcal{W}_{p(\cdot)}$. Moreover, $s_{\omega, \, p(\cdot)} \leq s_{\omega, \, q(\cdot)} + \frac{\alpha}{n}$.
\end{proposition}

\begin{proof}
By Definition \ref{pesos Wp}, $\omega \in \mathcal{W}_{q(\cdot)}$ if and only if
\[
(i)  \,\, \text{there exists} \,\, 0 < q_{\ast} < \min \{1, q_{-} \} \,\, \text{such that} \,\, 
\| \chi_Q \|_{L^{q(\cdot)/q_{\ast}}_{\omega^{q_{\ast}}}} < \infty, \,\,\, \text{and} 
\]
\[ 
\| \chi_Q \|_{L^{(q(\cdot)/q_{\ast})'}_{\omega^{-q_{\ast}}}} < \infty, \,\,\, \text{for all cube} \,\, Q;
\]
\[
(ii) \,\, \text{there exist $\kappa > 1$ and $s > 1$ such that Hardy-Littlewood maximal}
\]
\[ 
\text{operator $M$ is bounded on $L^{(sq(\cdot))'/\kappa}_{\omega^{-\kappa/s}}$}. 
\]
We define $\frac{1}{p_{\ast}} := \frac{1}{q_{\ast}} + \frac{\alpha}{n}$. Since 
$\frac{1}{p(\cdot)} - \frac{1}{q(\cdot)} = \frac{1}{p_{\ast}} - \frac{1}{q_{\ast}}$, it follows that $0 < p_{\ast} < \min \{1, p_{-} \}$ and
$\frac{p_{\ast}}{q_{\ast}} \left( \frac{p(\cdot)}{p_{\ast}} \right)' = \left( \frac{q(\cdot)}{q_{\ast}} \right)'$. So, from 
Lemma \ref{potencia s} and (i) above, we have
\[
\| \chi_Q \|_{L^{\left( \frac{p(\cdot)}{p_{\ast}} \right)'}_{\omega^{-p_{\ast}}}} = 
\| \chi_Q \|_{L^{\frac{p_{\ast}}{q_{\ast}} \left( \frac{p(\cdot)}{p_{\ast}} \right)'}_{\omega^{-q_{\ast}}}}^{\frac{p_{\ast}}{q_{\ast}}} =
\| \chi_Q \|_{L^{\left( \frac{q(\cdot)}{q_{\ast}} \right)'}_{\omega^{-q_{\ast}}}}^{\frac{p_{\ast}}{q_{\ast}}} < \infty.
\]
Being $p(\cdot) \leq q(\cdot)$, by \cite[Corollary 2.48]{Uribe2}, Lemma \ref{potencia s} and (i), we have
\[
\| \chi_{Q} \|_{L^{p(\cdot)/p_{\ast}}_{\omega^{p_{\ast}}}}^{q_{\ast}/p_{\ast}} = 
\| \chi_Q \, \omega^{p_{\ast}} \|_{L^{p(\cdot)/p_{\ast}}}^{q_{\ast}/p_{\ast}} \leq (1+|Q|)^{q_{\ast}/p_{\ast}}
\| \chi_Q \, \omega^{p_{\ast}} \|_{L^{q(\cdot)/p_{\ast}}}^{q_{\ast}/p_{\ast}}
\]
\[
= (1+|Q|)^{q_{\ast}/p_{\ast}} \| \chi_{Q} \|_{L^{q(\cdot)/q_{\ast}}_{\omega^{q_{\ast}}}} < \infty.
\]

From $(ii)$ follows that the maximal operator $M$ is bounded on $L^{(sq(\cdot))'}_{\omega^{-1/s}}$, with $s > 1$.
We fix $r > s + \frac{\alpha}{n}$, and define $q_0 := \frac{r}{s}$ and $\frac{1}{p_0} := \frac{1}{q_0} + \frac{\alpha}{nr}$. It is clear that
$\frac{r}{p_0} > 1$, $\frac{q_0}{p_0} > 1$ and 
$\frac{p_0}{q_0} \left( \frac{r}{p_0} p(\cdot) \right)' = \left( \frac{r}{q_0} q(\cdot) \right)'  = (s q(\cdot))'$. Thus, for 
$\tilde{s} := \frac{r}{p_0}$ and $\tilde{\kappa} := \frac{q_0}{p_0}$, we obtain 
\[
\| Mf \|_{L^{(\tilde{s}p(\cdot))'/\tilde{\kappa}}_{\omega^{-\tilde{\kappa}/\tilde{s}}}} = \| Mf \|_{L^{(sq(\cdot))'}_{\omega^{-1/s}}}
\leq C \| f \|_{L^{(sq(\cdot))'}_{\omega^{-1/s}}} = C \| f \|_{L^{(\tilde{s}p(\cdot))'/\tilde{\kappa}}_{\omega^{-\tilde{\kappa}/\tilde{s}}}}.
\]
So, $\omega \in \mathcal{W}_{p(\cdot)}$ and $s_{\omega, \, p(\cdot)} \leq s_{\omega, \, q(\cdot)} + \frac{\alpha}{n}$.
\end{proof}

The following necessary condition is due to Cruz-Uribe, Fiorenza and Neugebauer (see \cite[Theorem 1.5]{Uribe1}). It should be compared to the Muckenhoupt $\mathcal{A}_p$ condition from the study of weighted norm inequalities (see \cite[Chapter 7]{Grafakos}).

\begin{lemma} \label{Ap(.) cond}
Given a weight $\omega$ and an exponent function $p(\cdot) : \mathbb{R}^{n} \to (1, \infty)$ such that the Hardy-Littlewood maximal operator is bounded on $L^{p(\cdot)}_{w}(\mathbb{R}^{n})$, then, there exists a positive constant $C$ such that for every cube $Q \subset \mathbb{R}^{n}$,
\begin{equation} \label{Ap var condition}
\| \chi_{Q} \|_{L^{p(\cdot)}_\omega} \| \chi_{Q} \|_{L^{p'(\cdot)}_{\omega^{-1}}} \leq C |Q|.
\end{equation}
\end{lemma}

\begin{definition} \label{def Ap var} Given an exponent function $p(\cdot) : \mathbb{R}^{n} \to (1, \infty)$ and a weight $\omega$,
we write $\omega \in \mathcal{A}_{p(\cdot)}$ if $\omega$ satisfies (\ref{Ap var condition}).
\end{definition}

\begin{lemma} \label{2Q}
Let $p(\cdot) : \mathbb{R}^{n} \to (0, \infty)$ be a measurable function with $0 < p_{-} \leq p_{+} < \infty$. If 
$\omega \in \mathcal{W}_{p(\cdot)}$, then, for every cube $Q \subset \mathbb{R}^{n}$,
\[
\| \chi_{2Q} \|_{L^{p(\cdot)}_\omega}  \approx \| \chi_{Q} \|_{L^{p(\cdot)}_\omega} .
\]
\end{lemma}

\begin{proof} 
By the order preserving property of the norm $\| \, \cdot \, \|_{L^{p(\cdot)}_\omega}$, we have that 
\begin{equation} \label{chi2}
\| \chi_{Q} \|_{L^{p(\cdot)}_{\omega}} \leq \| \chi_{2Q} \|_{L^{p(\cdot)}_{\omega}}.
\end{equation}

On the other hand, since $\omega \in \mathcal{W}_{p(\cdot)}$, the maximal operator is bounded on $L^{(sp(\cdot))'}_{\omega^{-1/s}}$. 
Then, by Lemma \ref{Ap(.) cond} (exchanging the roles of $p(\cdot)$ and $p'(\cdot)$), (\ref{chi2}) above, and H\"older's inequality 
applied to $|Q| = \int \chi_Q(x) dx$ (see Lemma \ref{Holder ineq}), we have that
\[
\| \chi_{2Q} \|^{1/s}_{L^{p(\cdot)}_{\omega}} = \| \chi_{2Q} \|_{L^{sp(\cdot)}_{\omega^{1/s}}} \leq C |Q| 
\| \chi_{2Q} \|^{-1}_{L^{(sp(\cdot))'}_{\omega^{-1/s}}} \leq C |Q| 
\| \chi_{Q} \|^{-1}_{L^{(sp(\cdot))'}_{\omega^{-1/s}}} 
\]
\[
\leq C \| \chi_{Q} \|_{L^{sp(\cdot)}_{\omega^{1/s}}} = C
\| \chi_{Q} \|^{1/s}_{L^{p(\cdot)}_{\omega}}.
\]
This completes the proof.
\end{proof}

A weight $\omega$ satisfies the {\it reverse H\"older inequality} with exponent $s > 1$, denoted by $\omega \in RH_{s}$, if there exists a constant $C> 0$ such that for every cube $Q$,
$$\left(\frac{1}{|Q|} \int_{Q} [\omega(x)]^{s} dx \right)^{\frac{1}{s}} \leq C \frac{1}{|Q|} \int_{Q} \omega(x) dx;$$
the best possible constant is denoted by $[\omega]_{RH_s}$. We observe that if $\omega \in RH_s$, then by H\"older's inequality, 
$\omega \in RH_t$ for all $1 < t < s$, and $[\omega]_{RH_t} \leq [\omega]_{RH_s}$. 

\begin{lemma} \label{equiv norm}
Let $p(\cdot) \in \mathcal{P}^{\log}(\mathbb{R}^{n})$ with $0 < p_{-} \leq p_{+} < \infty$, $\gamma \in \mathbb{R}$, and let 
$f \in L^{1}_{loc}(\mathbb{R}^{n})$ be a function such that $|f(x)| \lesssim (1+|x|)^{\gamma}$. Then 
$\| f \|_{L^{p(\cdot)}(\mathbb{R}^{n})} \approx \| f \|_{L^{p_{\infty}}(\mathbb{R}^{n})}$.
\end{lemma}

\begin{proof}
We take $s > 1/p_{-}$. Then, by applying \cite[Lemma 2.7]{Diening} with $\omega \equiv 1$, we get
$\| f \|^{1/s}_{L^{p(\cdot)}} = \| |f|^{1/s} \|_{L^{sp(\cdot)}} \approx \| |f|^{1/s} \|_{L^{sp_{\infty}}} = \| f \|^{1/s}_{L^{p_{\infty}}}$.
\end{proof}

\begin{lemma} \label{estim w_p menos}
Let $p(\cdot) \in \mathcal{P}^{\log}(\mathbb{R}^{n})$ with $1 < p_{-} \leq p_{+} < \infty$ and $p_{\infty} = p_{-}$, and let $\omega$ be a weight such that $\omega(x)  \lesssim (1+|x|)^{\gamma}$ for some $\gamma \in \mathbb{R}$. If $\omega \in RH_{p_{+}}$, then for every cube 
$Q$ we have
\begin{equation} \label{estim pesada1}
\| \chi_{Q} \|_{L^{p(\cdot)}_{\omega}} \approx \left\{\begin{array}{c}
                                        \,\,\, \left[ \omega^{p_{-}}(Q) \right]^{1/p_{-}}, \hspace{1cm} \text{if} \,\, \ell(Q) > 1 \\\\
                                        \left[ \omega^{p_{-}(Q)}(Q) \right]^{1/p_{-}(Q)}, \,\, \text{if} \,\, \ell(Q) \leq 1 
                                                    \end{array} \right.. 
\end{equation}																			
\end{lemma}

\begin{proof}
Applying Lemma \ref{equiv norm}, with $f = \chi_Q \, \omega$, we obtain 
\begin{equation} \label{estim Hasto1}
\| \chi_Q \|_{L^{p(\cdot)}_\omega} = \| \chi_Q \, \omega \|_{L^{p(\cdot)}} \approx \left[ \omega^{p_{-}}(Q) \right]^{1/p_{-}},
\end{equation} 
for all cube $Q$.

From (\ref{estim Hasto1}) and since $1 < p_{-} \leq p_{-}(Q)$, by H\"older's inequality, we have
\[
\| \chi_Q \|_{L^{p(\cdot)}_{\omega}}^{p_{-}} \approx \int_{Q} [\omega(x)]^{p_{-}} dx \leq 
\left[\omega^{ p_{-}(Q)}(Q) \right]^{\frac{ p_{-}}{p_{-}(Q)}} \| \chi_Q \|_{\left(\frac{p_{-}(Q)}{p_{-}}\right)'} \leq 
\left[\omega^{ p_{-}(Q)}(Q) \right]^{\frac{ p_{-}}{p_{-}(Q)}}, 
\]
if $\ell(Q) \leq 1$. So,
\begin{equation} \label{estim Hasto2}
\| \chi_Q \|_{L^{p(\cdot)}_{\omega}} \lesssim \left[\omega^{ p_{-}(Q)}(Q) \right]^{\frac{1}{p_{-}(Q)}}, 
\end{equation}
holds for every cube $Q$ with $\ell(Q) \leq 1$.

On the other hand, $\omega \in RH_{p_{+}}$ and $p_{-}(Q) \leq p_{+}$, then $\omega \in RH_{p_{-}(Q)}$ and
\[
\left( \frac{1}{|Q|} \int_{Q} [\omega(x)]^{p_{-}(Q)} dx \right)^{1/p_{-}(Q)} \leq C |Q|^{-1} \int_{Q} \omega(x) dx \leq
C |Q|^{-1} \| \chi_Q \, \omega \|_{L^{p(\cdot)}} \| \chi_{Q} \|_{L^{p'(\cdot)}},
\]
\cite[Lemma 2.2]{Nakai} gives $\| \chi_{Q} \|_{L^{p'(\cdot)}} \approx |Q|^{1/p'_{+}(Q)}$ if $\ell(Q) \leq 1$, being 
$p'_{+}(Q) = (p_{-}(Q))'$, we obtain
\begin{equation} \label{estim Hasto3}
\left[ \omega^{p_{-}(Q)}(Q) \right]^{1/p_{-}(Q)} \lesssim \| \chi_Q \|_{L^{p(\cdot)}_\omega}, \,\,\,\, \text{if} \,\,\, \ell(Q) \leq 1.
\end{equation}
Finally, (\ref{estim Hasto1}), (\ref{estim Hasto2}) and (\ref{estim Hasto3}) give (\ref{estim pesada1}).
\end{proof}

\begin{corollary} \label{positive weights}
Let $p(\cdot) \in \mathcal{P}^{\log}(\mathbb{R}^{n})$ with $0 < p_{-} \leq p_{+} < \infty$ and $p_{\infty} = p_{-}$, and let 
$s > 1/p_{-}$. If $\omega(x)  \lesssim (1+|x|)^{\gamma}$ for some $\gamma \in \mathbb{R}$ and $\omega^{1/s} \in RH_{sp_{+}}$, then
(\ref{estim pesada1}) holds.
\end{corollary}

\begin{proof} Since $(sp(\cdot))_{+} = sp_{+}$, $(sp(\cdot))_{-} = sp_{-}$ and $(sp(\cdot))_{-}(Q) = sp_{-}(Q)$, 
Lemma \ref{estim w_p menos}, applied to $sp(\cdot)$ and $\omega^{1/s}$, gives
\[
\| \chi_Q \|_{L^{p(\cdot)}_\omega}^{1/s} = \| \chi_Q \|_{L^{sp(\cdot)}_{\omega^{1/s}}} \approx \left\{\begin{array}{c}
                                        \,\,\, \left[ \omega^{p_{-}}(Q) \right]^{1/sp_{-}}, \hspace{1cm} \text{if} \,\, \ell(Q) > 1 \\\\
                                        \left[ \omega^{p_{-}(Q)}(Q) \right]^{1/sp_{-}(Q)}, \,\, \text{if} \,\, \ell(Q) \leq 1 
                                                    \end{array} \right.. 
\]
So, the lemma follows.
\end{proof}

\begin{lemma} \label{estim w_p mas}
Let $p(\cdot) \in \mathcal{P}^{\log}(\mathbb{R}^{n})$ with $1 < p_{-} \leq p_{+} < \infty$ and $p_{\infty} = p_{+}$. 
If $\omega \in RH_{p_{+}}$, then for every cube $Q$ we have
\begin{equation} \label{estim pesada2}
\| \chi_{Q} \|_{L^{p(\cdot)}_{\omega}} \approx \left\{\begin{array}{c}
                                        \,\,\, \left[ \omega^{p_{+}}(Q) \right]^{1/p_{+}}, \hspace{1cm} \text{if} \,\, \ell(Q) > 1 \\\\
                                        \left[ \omega^{p_{+}(Q)}(Q) \right]^{1/p_{+}(Q)}, \,\, \text{if} \,\, \ell(Q) \leq 1 
                                                    \end{array} \right.. 
\end{equation}																			
\end{lemma}

\begin{proof} By hypothesis $p_{\infty} = p_{+}$, thus there exists $C_{\infty}>0$ such that
\[
0 \leq p_{+} - p(x) \leq \frac{C_{\infty}}{\log(e+|x|)}, \,\,\, \text{for all} \,\, x \in \mathbb{R}^{n}.
\]
Let $\lambda = [\omega^{p_{+}}(Q)]^{1/p_{+}}$, by applying \cite[Lemma 2.8]{Uribe1} with $t = 2/p_{-}$, we have
\[
\int_{Q} \left( \frac{\omega(x)}{\lambda} \right)^{p(x)} dx \leq C_{2/p_{-}} \int_{Q} \left( \frac{\omega(x)}{\lambda} \right)^{p_{+}} dx +
\int_{Q} \frac{dx}{(e + |x|)^{2n}}
\]
\[
\leq C_{2/p_{-}} + \int_{\mathbb{R}^{n}} \frac{dx}{(e + |x|)^{2n}} =: M < \infty.
\]
So, from the definition of the "norm" $\| \cdot \|_{L^{p(\cdot)}_{\omega}}$ we have
\begin{equation} \label{wQ size1}
\| \chi_{Q} \|_{L^{p(\cdot)}_{\omega}} \leq \max\{ M^{1/p_{-}}, M^{1/p_{+}} \} [\omega^{p_{+}}(Q)]^{1/p_{+}}, \,\,\, 
\text{for all cube} \,\, Q.
\end{equation}
Now, if $\ell(Q) > 1$, the condition $\omega \in RH_{p_{+}}$, H\"older's inequality and \cite[Lemma 2.2]{Nakai} give
\[
\left( \frac{1}{|Q|} \int_Q [\omega(x)]^{p_{+}} dx \right)^{1/p_{+}} \leq \frac{C}{|Q|} \int_Q \omega(x) dx
\]
\[
\leq \frac{C}{|Q|} \| \chi_Q \, \omega \|_{L^{p(\cdot)}} \| \chi_Q \|_{L^{p'(\cdot)}} \leq 
\frac{C}{|Q|} \| \chi_Q \, \omega \|_{L^{p(\cdot)}} |Q|^{1- \frac{1}{p_{+}}}.
\]
Consequently,
\begin{equation} \label{wQ size2}
[\omega^{p_{+}}(Q)]^{1/p_{+}} \leq C \| \chi_Q \|_{L^{p(\cdot)}_{\omega}}, \,\,\, \text{if} \,\, \ell(Q) > 1.
\end{equation}
Then, (\ref{wQ size1}) and (\ref{wQ size2}) give
\begin{equation} \label{wQ size3}
\| \chi_Q \|_{L^{p(\cdot)}_{\omega}} \approx [\omega^{p_{+}}(Q)]^{1/p_{+}}, \,\,\, \text{if} \,\, \ell(Q) > 1.
\end{equation}

Now, we study the case $\ell(Q) \leq 1$. It is easy to check that
\[
0 \leq p_{+}(Q) - p(x) \leq \frac{C}{\log(e + |x|)}, \,\,\, \text{for all} \,\, x \in Q.
\]
Since $1 < p_{+}(Q) \leq p_{+}$, we have that $\omega \in RH_{p_{+}(Q)}$. Then, by reasoning as above, we get
\begin{equation} \label{wQ size4}
\| \chi_Q \|_{L^{p(\cdot)}_{\omega}} \approx \left[ \omega^{p_{+}(Q)}(Q) \right]^{1/p_{+}(Q)}, \,\,\, \text{if} \,\, \ell(Q) \leq 1.
\end{equation}
Finally, (\ref{wQ size3}) and (\ref{wQ size4}) give (\ref{estim pesada2}).
\end{proof}

\begin{corollary} \label{negative weights}
Let $p(\cdot) \in \mathcal{P}^{\log}(\mathbb{R}^{n})$ with $0 < p_{-} \leq p_{+} < \infty$ and $p_{\infty} = p_{+}$, and let 
$s > 1/p_{-}$. If $\omega^{1/s} \in RH_{sp_{+}}$, then (\ref{estim pesada2}) holds.
\end{corollary}

\begin{proof}
The proof follows from Lemmas \ref{potencia s} and \ref{estim w_p mas}.
\end{proof}

\section{Off-diagonal Fefferman-Stein inequality} 

\, We apply extrapolation techniques to obtain an "off-diagonal" version of the Fefferman-Stein vector-valued maximal inequality on 
$L^{p(\cdot)}_{\omega}$. The following result generalizes Theorem 3.1 obtained in \cite{Ho1}.

\begin{theorem} \label{Feff-Stein ineq}
Let $0 \leq \alpha < n$, $1 < u < \infty$ and let $q(\cdot) : \mathbb{R}^{n} \to (0, \infty)$ be a measurable function with 
$0 < q_{-} \leq q_{+} < \infty$. If $\omega \in \mathcal{W}_{q(\cdot)}$, then for 
$\frac{1}{p(\cdot)} := \frac{1}{q(\cdot)} + \frac{\alpha}{n}$ and any $r > s_{\omega, \, q(\cdot)} + \frac{\alpha}{n}$,
\begin{equation} \label{maximal fract ineq}
\left\| \left( \sum_{j \in \mathbb{N}} (M_{\frac{\alpha}{r}}f_j)^{u} \right)^{1/u} \right\|_{L^{rq(\cdot)}_{\omega^{1/r}}} \lesssim 
\left\| \left( \sum_{j \in \mathbb{N}} |f_j|^{u} \right)^{1/u} \right\|_{L^{rp(\cdot)}_{\omega^{1/r}}},
\end{equation}
holds for all sequences of bounded measurable functions with compact support $\{ f_j \}_{j=1}^{\infty}$.
\end{theorem}

\begin{proof}
Given $r > s_{\omega, \, q(\cdot)} + \frac{\alpha}{n}$, from the definition of $s_{\omega, \, q(\cdot)}$, we have 
$s > s_{\omega, \, q(\cdot)}$ such that $s + \frac{\alpha}{n} < r$ and $M$ is bounded on 
$L^{(sq(\cdot))'}_{w^{-1/s}}(\mathbb{R}^{n})$. Define
$$\mathcal{F} = \left\{ \left( \left( \sum_{j=1}^{K} (M_{\frac{\alpha}{r}}f_j)^{u} \right)^{1/u}, 
\left( \sum_{j=1}^{K} |f_j|^{u} \right)^{1/u} \right) : K \in \mathbb{N}, \{f_j \}_{j=1}^{K} \subset L^{\infty}_{comp} \right\},$$
where $L^{\infty}_{comp}$ denotes the set of bounded functions with compact support.

Let $q_0 = \frac{r}{s}$ and let $p_0$ be defined by $\frac{1}{p_0} := \frac{1}{q_0} + \frac{\alpha}{nr}$. Since 
$r > s + \frac{\alpha}{n}$ we have that $1 < p_0 < \frac{nr}{\alpha}$. From \cite[Theorem 3]{Muck} follows that there exists an universal constant $C > 0$ such that for any $(F, G) \in \mathcal{F}$ and any $v \in \mathcal{A}_1$ (for the definition of the $\mathcal{A}_{1}$ class, the reader may refer to \cite[Chapter 7]{Grafakos})
\begin{equation} \label{weighted fract ineq}
\int [F(x)]^{q_0} v(x) \, dx \leq C \left( \int [G(x)]^{p_0} [v(x)]^{p_0/q_0} \, dx \right)^{q_0/p_0},
\end{equation}
since $v \in \mathcal{A}_1$ implies that $v^{1/q_0} \in \mathcal{A}_{p_0, q_0}$ (for the definition of the $\mathcal{A}_{p,q}$ class, see \cite[inequality (1.1)]{Muck}). On the other hand, by Proposition \ref{norma equivalente}, we have
\begin{equation} \label{norma}
\| F \|^{q_0}_{L^{rq(\cdot)}_{\omega^{1/r}}} = \| F^{q_0} \|_{L^{sq(\cdot)}_{\omega^{1/s}}}
\leq C \sup \left\{ \int_{\mathbb{R}^{n}} \left| [F(x)]^{q_0} g(x) \right| dx : \| g \|_{L^{(sq(\cdot))'}_{\omega^{-1/s}}} \leq 1 \right\}
\end{equation}
for some constant $C > 0$.

Let $\mathcal{R}$ be the operator defined on $L^{(sq(\cdot))'}_{\omega^{-1/s}}$ by
\[
\mathcal{R}g(x) = \sum_{k=0}^{\infty} \frac{M^{k}g(x)}{2^{k} \| M \|_{L^{(sq(\cdot))'}_{\omega^{-1/s}}}^{k}},
\]
where, for $k \geq 1$, $M^{k}$ denotes $k$ iterations of the Hardy-Littlewood maximal operator $M$, $M^{0} = M$, and 
$\| M \|_{L^{(sq(\cdot))'}_{\omega^{-1/s}}}$ is the operator norm of the maximal operator $M$ on 
$L^{(sq(\cdot))'}_{\omega^{-1/s}}$. It follows immediately from this definition that:

$(i)$ if $g$ is non-negative, $g(x) \leq \mathcal{R}g(x)$ a.e. $x \in \mathbb{R}^{n}$;

$(ii)$ $\| \mathcal{R}g \|_{L^{(sq(\cdot))'}_{\omega^{-1/s}}} \leq 2 \| g \|_{L^{(sq(\cdot))'}_{\omega^{-1/s}}}$; 

$(iii)$ $\mathcal{R}g \in \mathcal{A}_1$ with $[\mathcal{R}g]_{\mathcal{A}_1} \leq 2 \| M \|_{L^{(sq(\cdot))'}_{\omega^{-1/s}}}$.
\\
Since $F$ is non-negative, we can take the supremum in (\ref{norma}) over those non-negative $g$ only. For any fixed non-negative 
$g \in L^{(sq(\cdot))'}_{\omega^{-1/s}}$, by $(i)$ above we have that
\begin{equation} \label{int g}
\int [F(x)]^{q_0} g(x) dx \leq \int [F(x)]^{q_0} (\mathcal{R}g)(x) dx.
\end{equation}
Then $(iii)$ and (\ref{weighted fract ineq}), and H\"older's inequality yield
\begin{equation} \label{int Rg}
\int [F(x)]^{q_0} (\mathcal{R}g)(x) dx \leq C \left( \int [G(x)]^{p_0} [(\mathcal{R}g)(x)]^{p_0 / q_0} dx \right)^{q_0/p_0} 
\end{equation}
\[
\leq C \| G^{p_0} \|_{L^{rp(\cdot)/p_0}_{\omega^{p_0/r}}}^{q_0/p_0} 
\|(\mathcal{R}g)^{p_0/q_0} \|_{L^{(rp(\cdot)/p_0)'}_{\omega^{-p_0/r}}}^{q_0/p_0}
\]
\[
= C \| G \|^{q_0}_{L^{rp(\cdot)}_{\omega^{1/r}}} 
\|\mathcal{R}g \|_{L^{\frac{p_0}{q_0} \left(\frac{rp(\cdot)}{p_0} \right)'}_{\omega^{-q_0/r}}}
\]
being $\frac{1}{rp(\cdot)} - \frac{1}{rq(\cdot)} = \frac{1}{p_0} - \frac{1}{q_0}$ and $q_0 = \frac{r}{s}$, we have 
$\frac{p_0}{q_0} \left( \frac{r}{p_0} p(\cdot) \right)' = \left( \frac{r}{q_0} q(\cdot) \right)' = (sq(\cdot))'$, so
\[
= C \| G \|^{q_0}_{L^{rp(\cdot)}_{\omega^{1/r}}} \| \mathcal{R}g \|_{L^{(sq(\cdot))'}_{\omega^{-1/s}}}
\]
now, $(ii)$ gives
\[
\leq C \| G \|^{q_0}_{L^{rp(\cdot)}_{\omega^{1/r}}} \| g \|_{L^{(sq(\cdot))'}_{\omega^{-1/s}}}.
\]
Thus, (\ref{int g}) and (\ref{int Rg}) lead to
\begin{equation} \label{norma2}
\int [F(x)]^{q_0} g(x) dx \leq C \| G \|^{q_0}_{L^{rp(\cdot)}_{\omega^{1/r}}},
\end{equation}
for all non-negative $g$ with $\| g \|_{L^{(sq(\cdot))'}_{\omega^{-1/s}}} \leq 1$. Then, (\ref{norma}) and (\ref{norma2}) give 
(\ref{maximal fract ineq}) for all finite sequences $\{f_j \}_{j=1}^{K} \subset L^{\infty}_{comp}$. Finally, by passing to the limit, we obtain (\ref{maximal fract ineq}) for all infinite sequences $\{f_j \}_{j=1}^{\infty} \subset L^{\infty}_{comp}$.
\end{proof}

\begin{corollary} \label{Estimate_qp}
Let $0 < \alpha < n$, $q(\cdot) : \mathbb{R}^{n} \to (0, \infty)$ be a measurable function with 
$0 < q_{-} \leq q_{+} < \infty$ and $\omega \in \mathcal{W}_{q(\cdot)}$. If $\frac{1}{p(\cdot)} := \frac{1}{q(\cdot)} + \frac{\alpha}{n}$ 
and $\| \chi_Q \|_{L^{q(\cdot)}_{\omega}} \approx |Q|^{-\alpha/n} \| \chi_Q \|_{L^{p(\cdot)}_{\omega}}$ for every cube $Q$, then 
for any sequence of scalars $\{ \lambda_j \}_{j \in \mathbb{N}}$, any family of cubes $\{ Q_j \}_{j \in \mathbb{N}}$, and 
any $\theta \in (0, \infty)$ fixed we have
\[
\left\| \sum_{j \in \mathbb{N}} \left( \frac{|\lambda_j| \chi_{Q_j}}{\| \chi_{Q_j} \|_{L^{q(\cdot)}_{\omega}}} \right)^{\theta} 
\right\|^{1/\theta}_{L^{q(\cdot)/\theta}_{\omega^{\theta}}} \lesssim
\left\| \sum_{j \in \mathbb{N}} \left( \frac{|\lambda_j| \chi_{Q_j}}{\| \chi_{Q_j} \|_{L^{p(\cdot)}_{\omega}}} \right)^{\theta} 
\right\|^{1/\theta}_{L^{p(\cdot)/\theta}_{\omega^{\theta}}}.
\]
\end{corollary}

\begin{proof} Since $\| \chi_Q \|_{L^{q(\cdot)}_{\omega}} \approx |Q|^{-\alpha/n} \| \chi_Q \|_{L^{p(\cdot)}_{\omega}}$ for every cube $Q$ we obtain
\[
\left\| \sum_{j} \left( \frac{|\lambda_j| \chi_{Q_j}}{\| \chi_{Q_j} \|_{L^{q(\cdot)}_{\omega}}} \right)^{\theta} 
\right\|^{1/\theta}_{L^{q(\cdot)/\theta}_{\omega^{\theta}}} =
\left\| \left\{ \sum_{j} \left( \frac{|\lambda_j| \chi_{Q_j}}{\| \chi_{Q_j} \|_{L^{q(\cdot)}_{\omega}}} \right)^{\theta} \right\}^{1/\theta}
\right\|_{L^{q(\cdot)}_{\omega}}
\]
\[
\lesssim \left\| \left\{ \sum_{j} \left( \frac{|\lambda_j| |Q_j|^{\alpha/n} \chi_{Q_j}}{\| \chi_{Q_j} \|_{L^{p(\cdot)}_{\omega}}} \right)^{\theta} \right\}^{1/\theta}
\right\|_{L^{q(\cdot)}_{\omega}},
\]
it is easy to check that $|Q_j|^{\alpha/n} \chi_{Q_j} (x) \leq M_{\frac{\alpha \theta}{N}}(\chi_{Q_j})^{\frac{N}{\theta}}(x)$ for all $j$ and all
$N \in \mathbb{N}$, so
\[
\lesssim \left\| \left\{ \sum_{j} \left( \frac{|\lambda_j| M_{\frac{\alpha \theta}{N}}(\chi_{Q_j})^{\frac{N}{\theta}}}{\| \chi_{Q_j} \|_{L^{p(\cdot)}_{\omega}}} \right)^{\theta} \right\}^{1/\theta}
\right\|_{L^{q(\cdot)}_{\omega}} =
\left\| \left\{ \sum_{j} \left( \frac{|\lambda_j|^{\theta} M_{\frac{\alpha \theta}{N}}(\chi_{Q_j})^{N}}{\| \chi_{Q_j} \|^{\theta}_{L^{p(\cdot)}_{\omega}}} \right) \right\}^{1/\theta}
\right\|_{L^{q(\cdot)}_{\omega}}
\]
\[
= \left\| \left\{ \sum_{j} \left( \frac{|\lambda_j|^{\theta} M_{\frac{\alpha \theta}{N}}(\chi_{Q_j})^{N}}{\| \chi_{Q_j} \|^{\theta}_{L^{p(\cdot)}_{\omega}}} \right) \right\}^{1/N}
\right\|^{N/\theta}_{L^{Nq(\cdot)/\theta}_{\omega^{\theta/N}}},
\]
taking $N$ such that $N/\theta > s_{\omega, \, q(\cdot)} + \frac{\alpha}{n}$, by Theorem \ref{Feff-Stein ineq}, we get
\[
\lesssim \left\| \left\{ \sum_{j} \left( \frac{|\lambda_j|^{\theta} \chi_{Q_j}}{\| \chi_{Q_j} \|^{\theta}_{L^{p(\cdot)}_{\omega}}} \right) \right\}^{1/N}
\right\|^{N/\theta}_{L^{Np(\cdot)/\theta}_{\omega^{\theta/N}}} =
\left\| \sum_{j} \left( \frac{|\lambda_j| \chi_{Q_j}}{\| \chi_{Q_j} \|_{L^{p(\cdot)}_{\omega}}} \right)^{\theta}
\right\|^{1/\theta}_{L^{p(\cdot)/\theta}_{\omega^{\theta}}}.
\]
This completes the proof.
\end{proof}

\section{Weighted variable estimates for Riesz potential}

Let $0 < \alpha < n$. The \textit{Riesz potential} of order $\alpha$ is the fractional operator $I_{\alpha }$ defined by
\begin{equation}
I_{\alpha }f(x)=\int_{\mathbb{R}^{n}} f(y) |x-y|^{\alpha - n}dy, \,\,\, x \in \mathbb{R}^{n},
\label{Ia}
\end{equation}
$f \in \mathcal{S}(\mathbb{R}^{n})$. A well known result of Sobolev gives the
boundedness of $I_{\alpha }$ from $L^{p}(\mathbb{R}^{n})$ into $L^{q}(\mathbb{R}^{n})$ for $1 < p <\frac{n}{\alpha }$ and 
$\frac{1}{q}=\frac{1}{p}- \frac{\alpha }{n}.$ In \cite{Capone} C. Capone, D. Cruz Uribe and A. Fiorenza
extend this result to the case of Lebesgue spaces with variable exponents $L^{p(\cdot)}.$ In \cite{Stein2} E. Stein and G. Weiss used the theory of harmonic functions of several variables to prove that these operators are bounded
from $H^{1}(\mathbb{R}^{n})$ into $L^{\frac{n}{n-\alpha }}(\mathbb{R}^{n})$. In \cite{Taibleson}, M. Taibleson and G. Weiss obtained the boundedness of the Riesz potential $I_{\alpha }$ from the Hardy spaces $H^{p}(\mathbb{R}^{n})$ into $H^{q}(\mathbb{R}^{n}),$ for $0<p<1$ and $\frac{1}{q}=\frac{1}{p}-\frac{\alpha }{n}$. These results were extended to the context of variable Hardy spaces by M. Urciuolo and
the author in \cite{Rocha1, Rocha3}. 

In this section we will prove that the Riesz potential $I_{\alpha}$ is bounded from weighted variable Hardy spaces into weighted variable Lebesgue spaces. The main tools that we will use are Theorem \ref{Feff-Stein ineq}, Corollary \ref{Estimate_qp} and 
Theorem \ref{w atomic decomp}.

\begin{theorem} \label{weighted Hp-Lq}
Let $0 < \alpha < n$, $q(\cdot) \in \mathcal{P}^{\log}(\mathbb{R}^{n})$ with $0 < q_{-} \leq q_{+} < \infty$, and 
$\omega \in \mathcal{W}_{q(\cdot)}$. If $\frac{1}{p(\cdot)} := \frac{1}{q(\cdot)} + \frac{\alpha}{n}$ and 
$\| \chi_Q \|_{L^{q(\cdot)}_{\omega}} \approx |Q|^{-\alpha/n} \| \chi_Q \|_{L^{p(\cdot)}_{\omega}}$ for every cube $Q$, then the Riesz potential $I_{\alpha}$ given by (\ref{Ia}) can be extended to a bounded operator 
$H^{p(\cdot)}_{\omega}(\mathbb{R}^{n}) \to L^{q(\cdot)}_{\omega}(\mathbb{R}^{n})$.
\end{theorem}

\begin{proof}
Let $\omega \in \mathcal{W}_{q(\cdot)}$, by Definition \ref{pesos Wp}, there exists $0 < \theta < 1$ such that 
$\frac{1}{\theta} \in \mathbb{S}_{\omega, \, q(\cdot)}$. Now, we take $q_0 > \frac{n}{n - \alpha}$ such that 
$q_0 > \theta \left( \kappa_{\omega, \, q(\cdot)}^{1/\theta} \right)'$, and define  $\frac{1}{p_0} := \frac{1}{q_0} + \frac{\alpha}{n}$. 
By Proposition \ref{WqcWp}, for $\frac{1}{p(\cdot)} = \frac{1}{q(\cdot)} + \frac{\alpha}{n}$, we have that
$\mathcal{W}_{q(\cdot)} \subset \mathcal{W}_{p(\cdot)}$ and $s_{\omega, \, p(\cdot)} \leq s_{\omega, \, q(\cdot)} + \frac{\alpha}{n}$. So, given $f \in S_{0}(\mathbb{R}^{n})$, from Theorem \ref{w atomic decomp} (taking $s = p_0$) and since one can always choose
atoms with additional vanishing moment, we have that there exist a sequence of real numbers 
$\{\lambda_j\}_{j=1}^{\infty}$, a sequence of cubes $Q_j = Q(z_j,r_j)$ centered at $z_j$ with side length $r_j$ and 
$\omega - (p(\cdot), p_0, \lfloor n s_{\omega, \, q(\cdot)} + \alpha - n \rfloor)$ atoms $a_j$ supported on $Q_j$, satisfying
\begin{equation} \label{ineq Hpw}
\left\| \sum_{j} \left( \frac{|\lambda_j |}{\| \chi_{Q_j} \|_{L^{p(\cdot)}_{\omega}}} \right)^{\theta} \chi_{Q_j} \right\|_{L^{p(\cdot)/\theta}_{\omega^{\theta}}}^{1/\theta} \lesssim \| f \|_{H^{p(\cdot)}_{\omega}},
\end{equation}
and $f = \sum_{j} \lambda_j a_j$ converges in $L^{p_0}(\mathbb{R}^{n})$. By Sobolev's Theorem we have that 
$I_{\alpha }$ is bounded from  $L^{p_{0}}\left( \mathbb{R}^{n}\right) $ into $L^{q_{0}}\left( \mathbb{R}^{n}\right)$, so
\[
|I_{\alpha}f(x)| \leq \sum_{j} |\lambda_j| |I_{\alpha}a_j(x)|, \,\,\,\,\, \textit{a.e.} \, x \in \mathbb{R}^{n}.
\]
Then
\begin{equation} \label{J1_J2}
\| I_{\alpha}f \|_{L^{q(\cdot)}_{\omega}} \lesssim \left\| \sum_{j} |\lambda_j| \, \chi_{2 Q_j} \cdot I_{\alpha} a_j 
\right\|_{L^{q(\cdot)}_{\omega}} + \left\| \sum_{j} |\lambda_j| \, \chi_{\mathbb{R}^{n} \setminus 2 Q_j} \cdot I_{\alpha} a_j 
\right\|_{L^{q(\cdot)}_{\omega}} = J_1 + J_2,
\end{equation}
where $2Q_j = Q(z_j, 2r_j)$. To estimate $J_1$, we again apply Sobolev's Theorem and obtain
\[
\left\| (I_{\alpha} a_j)^{\theta} \right\|_{L^{q_{0}/\theta}(2Q_{j})} = 
\left\| I_{\alpha} a_j \right\|_{L^{q_{0}}(2Q_{j})}^{\theta} \lesssim \left\| a_j \right\|_{L^{p_{0}}}^{\theta} \lesssim 
\frac{| Q_j |^{\frac{\theta}{p_{0}}}}{\left\| \chi _{Q_j }\right\|_{L^{p(\cdot)/\theta}_{\omega^{\theta}}}} \lesssim 
\frac{\left| 2Q_{j} \right|^{\frac{\theta}{q_{0}}}}{\left\| \chi_{2Q_{j}} \right\|_{L^{q(\cdot)/\theta}_{\omega^{\theta}}}},
\]
where $\theta$, $q_0$ and $p_0$ are given above. The last inequality follows from the condition $\| \chi_Q \|_{L^{q(\cdot)}_{\omega}} 
\approx |Q|^{-\alpha/n} \| \chi_Q \|_{L^{p(\cdot)}_{\omega}}$ assumed for every cube $Q$ and Lemma \ref{2Q} applied to the exponent $q(\cdot)$. Since $\frac{1}{\theta} \in \mathbb{S}_{\omega, \, q(\cdot)}$ and $q_0 > \theta \left( \kappa_{\omega, \, q(\cdot)}^{1/\theta} \right)'$, we apply the $\theta$-inequality and Proposition \ref{bk functions} with $b_j = \left( \chi_{2 Q_j} \cdot I_{\alpha}(a_j) \right)^{\theta}$ and 
$A_j = \left\| \chi_{2Q_{j}} \right\|_{L^{q(\cdot)/\theta}_{\omega^{\theta}}}^{-1}$ to obtain
\[
J_1 \lesssim \left\| \sum_{j} \left(|\lambda_j| \, \chi_{2 Q_j} \cdot I_{\alpha} a_j \right)^{\theta} \right\|^{1/\theta}_{L^{q(\cdot)/\theta}_{\omega^{\theta}}} \lesssim \left\| \sum_{j} \left( \frac{|\lambda_j|}{\left\| \chi_{2Q_{j}} \right\|_{L^{q(\cdot)}_{\omega}}} 
\right)^{\theta} \chi_{2 Q_j} \right\|^{1/\theta}_{L^{q(\cdot)/\theta}_{\omega^{\theta}}}.
\]
Being $\chi_{2 Q_j} \leq M(\chi_{Q_j})^{2}$, by Theorem \ref{Feff-Stein ineq}, we have
\[
J_1 \lesssim \left\| \left\{ \sum_{j} \left( \frac{|\lambda_j|^{\theta/2}}{\left\| \chi_{2Q_{j}} \right\|^{\theta/2}_{L^{q(\cdot)}_{\omega}}}M(\chi_{Q_j})\right)^{2}  \right\}^{1/2} \right\|^{2/\theta}_{L^{2q(\cdot)/\theta}_{\omega^{\theta/2}}} \lesssim
\left\| \sum_{j} \left( \frac{|\lambda_j |}{\| \chi_{Q_j} \|_{L^{q(\cdot)}_{\omega}}} \right)^{\theta} \chi_{Q_j} 
\right\|_{L^{q(\cdot)/\theta}_{\omega^{\theta}}}^{1/\theta}.
\]
Corollary \ref{Estimate_qp} and (\ref{ineq Hpw}) give
\begin{equation} \label{J1}
J_1 \lesssim \left\| \sum_{j} \left( \frac{|\lambda_j |}{\| \chi_{Q_j} \|_{L^{p(\cdot)}_{\omega}}} \right)^{\theta} \chi_{Q_j} 
\right\|_{L^{p(\cdot)/\theta}_{\omega^{\theta}}}^{1/\theta} \lesssim \| f\|_{H^{p(\cdot)}_{\omega}}.
\end{equation}

Now, we estimate $J_2$. Let $d = \lfloor n s_{\omega, \, q(\cdot)} + \alpha - n \rfloor$, and let $a_j(\cdot)$ be a 
$\omega - (p(\cdot), p_0, d)$ atom supported on the cube $Q_j = Q(z_j, r_j)$. In view of the moment condition of $a_j(\cdot)$ we obtain
\[
I_{\alpha}a_j(x) = \int_{Q_j} \left( |x-y|^{\alpha- n} - q_{d}(x,y) \right) a(y) dy, \,\,\,\,\,\, \textit{for all} \,\,\, x \notin 2Q_j,
\]
where $q_{d}$ is the degree $d$ Taylor polynomial of the function $y \rightarrow |x-y|^{\alpha - n}$ expanded around $z_j$. By the standard estimate of the remainder term of the Taylor expansion, there exists $\xi$ between  $y$ and $z_j$ such that
\[
\left| |x-y|^{\alpha- n} - q_{d}(x,y) \right| \leq C |y - z_j|^{d +1} |x - \xi|^{-n+ \alpha -d -1},
\]
for any $y  \in Q_j$ and any $x \notin 2Q_j$, since $|x - \xi| \geq \displaystyle{\frac{|x- z_j|}{1+\sqrt{n}}}$, we get
\[
\left||x-y|^{\alpha- n} - q_{d}(x,y) \right| \leq C r^{d +1} |x - z_j|^{-n+ \alpha -d -1},
\]
this inequality and the condition $(a2)$ allow us to conclude that
\begin{equation} \label{Ialfa}
|I_{\alpha}a_j(x) | \lesssim \| a_j \|_{1} \, r^{d +1} |x-z_j|^{-n + \alpha - d -1} 
\end{equation}
\[
\lesssim |Q_j|^{1- \frac{1}{p_0}} \| a_j \|_{p_0} r^{d +1} |x - z_j|^{-n+ \alpha -d -1} \lesssim
\frac{r^{n+d+1}}{\| \chi_{Q_j} \|_{L^{p(\cdot)}_\omega}} |x-z_j|^{-n + \alpha - d -1}
\]
\[ 
\lesssim \frac{\left[ M_{\frac{\alpha n}{n+d+1}}(\chi_{Q_j}) (x) \right]^{\frac{n+d+1}{n}}}{\| \chi_{Q_j} \|_{L^{p(\cdot)}_\omega}}, \,\,\,\,\,\, for \,\, all \,\,\,\, x \notin 2Q_j.
\]
We put $r = \frac{n+d+1}{n}$, thus (\ref{Ialfa}) leads to
\[
J_2 \lesssim \left\| \left\{ \sum_j  \frac{|\lambda_j|}{\| \chi_{Q_j} \|_{L^{p(\cdot)}_\omega}} \left[ M_{\frac{\alpha}{r}}(\chi_{Q_j}) \right]^{r} \right\}^{1/r} \right\|^{r}_{L^{rq(\cdot)}_{\omega^{1/r}}}. 
\]
Since
\[
r = \frac{n + \lfloor n s_{\omega, \, q(\cdot)} + \alpha - n \rfloor + 1}{n} > s_{\omega, \, q(\cdot)} + \frac{\alpha}{n},
\]
to apply Theorem \ref{Feff-Stein ineq}, with $u=r$, we obtain
\[
J_2 \lesssim \left\| \left\{ \sum_j  \frac{|\lambda_j|}{\| \chi_{Q_j} \|_{L^{p(\cdot)}_\omega}} \chi_{Q_j} \right\}^{1/r} 
\right\|^{r}_{L^{rp(\cdot)}_{\omega^{1/r}}} = \left\| \sum_j \frac{|\lambda_j|}{\| \chi_{Q_j} \|_{L^{p(\cdot)}_\omega}} \chi_{Q_j}
\right\|_{L^{p(\cdot)}_{\omega}}. 
\]
Being $0 < \theta < 1$, the $\theta$-inequality and (\ref{ineq Hpw}) give
\begin{equation} \label{J2}
J_2 \lesssim \left\| \sum_j \left(\frac{|\lambda_j|}{\| \chi_{Q_j} \|_{L^{p(\cdot)}_\omega}} \right)^{\theta} \chi_{Q_j}
\right\|^{1/\theta}_{L^{p(\cdot)/\theta}_{\omega^{\theta}}} \lesssim \| f \|_{H^{p(\cdot)}_{\omega}}.
\end{equation}
Hence, (\ref{J1_J2}), (\ref{J1}) and (\ref{J2}) yield
\[
\| I_{\alpha} f \|_{L^{q(\cdot)}_{\omega}} \lesssim \| f \|_{H^{p(\cdot)}_{\omega}}, \,\,\, \text{for all} \,\, 
f \in \mathcal{S}_{0}(\mathbb{R}^{n}).
\]

Finally, by taking into account that $p(\cdot) \in \mathcal{P}^{\log}(\mathbb{R}^{n})$, $0 < p_{-} \leq p_{+} < \infty$ and 
$\omega \in \mathcal{W}_{p(\cdot)}$, the theorem follows from Theorem \ref{dense}.
\end{proof}

\section{Example: power weights}

For $\gamma > -n$, let $w_{\gamma}(x) = |x|^{\gamma}$ with $x \in \mathbb{R}^{n} \setminus \{0\}$. From the estimates 
obtained in \cite[Example 7.1.6, pp. 505]{Grafakos}, we observe that 
\begin{equation} \label{RH1}
\omega_{\gamma} \in \bigcap_{t > 1} RH_{t}, \,\,\, \text{if} \,\, \gamma \geq 0
\end{equation}
and
\begin{equation} \label{RH2}
\omega_{\gamma} \in \bigcap_{1 < t < n/|\gamma|} RH_{t}, \,\,\, \text{if} \,\, -n < \gamma < 0.
\end{equation}
Moreover, for $0 \leq \alpha < n$, 
$0 < p < \frac{n}{\alpha}$, $\frac{1}{q} = \frac{1}{p} - \frac{\alpha}{n}$ and each cube $Q$ we have that
\begin{equation} \label{positive gamma}
[\omega_{\gamma}^{p}(Q)]^{-1/p} [\omega_{\gamma}^{q}(Q)]^{1/q} = [\omega_{\gamma p}(Q)]^{-1/p} [\omega_{\gamma q}(Q)]^{1/q}
\approx |Q|^{-\alpha/n}, \,\,\, \text{for every} \,\, \gamma \geq 0
\end{equation}
and 
\[
[\omega_{\gamma}^{p}(Q)]^{-1/p} [\omega_{\gamma}^{q}(Q)]^{1/q} \approx |Q|^{-\alpha/n}, \,\,\, \text{for every} \,\, -n < \gamma < 0
\,\, \text{with} \,\, q < \frac{n}{|\gamma|}.
\]

Let $q(\cdot) \in \mathcal{P}^{\log}(\mathbb{R}^{n})$ with $0 < q_{\infty} = q_{-} \leq q_{+} < \infty$, and let $p(\cdot)$ be defined 
by $\frac{1}{p(\cdot)} := \frac{1}{q(\cdot)} + \frac{\alpha}{n}$. If $\gamma \geq 0 $, then, by (\ref{positive gamma}) and 
Corollary \ref{positive weights} (first applied to $q(\cdot)$ and then to $p(\cdot)$), we get
\[
\| \chi_Q \|_{L^{q(\cdot)}_{\omega_\gamma}} \approx |Q|^{-\alpha/n} \| \chi_Q \|_{L^{p(\cdot)}_{\omega_{\gamma}}}.
\]
It is clear that $q'(\cdot) \in \mathcal{P}^{\log}(\mathbb{R}^{n})$, $0 < q'_{-} \leq q'_{+} = q'_{\infty} < \infty$, 
$q'_{+} = (q_{-})'$ and $q'_{+}(Q) = (q_{-}(Q))'$. Then for every $0 \leq \gamma < \frac{n}{q'_{+}}$, by (\ref{RH1}) and 
Corollary \ref{positive weights}; and (\ref{RH2}) and Corollary \ref{negative weights} 
(applied to $\omega_{\gamma}$ and $q(\cdot)$; and $\omega_{-\gamma}$ and  $q'(\cdot)$ respectively), we obtain that
\[
\| \chi_Q \|_{L^{q(\cdot)}_{\omega_\gamma}} \| \chi_Q \|_{L^{(q(\cdot))'}_{\omega_{-\gamma}}} \approx
\left\{\begin{array}{c}
\,\,\, \left[ \omega_{\gamma}^{q_{-}}(Q) \right]^{1/q_{-}} [\omega_{-\gamma}^{q'_{+}}(Q)]^{1/q'_{+}}, \hspace{1cm} \text{if} \,\, 
\ell(Q) > 1 \\\\
\left[ \omega_{\gamma}^{q_{-}(Q)}(Q) \right]^{1/q_{-}(Q)} \left[\omega_{-\gamma}^{q'_{+}(Q)}(Q) \right]^{1/q'_{+}(Q)}, \,\, \text{if} \,\, 
\ell(Q) \leq 1 
                                                    \end{array} \right\} \approx |Q|, 
\]
for all cube $Q \subset \mathbb{R}^{n}$, where the last estimate follows from \cite[Example 7.1.6, pp. 505]{Grafakos}. In particular, for 
$s > 1/q_{-}$ and $0 \leq \gamma < \frac{sn}{(sq(\cdot))'_{+}}$, we have
\[
\| \chi_Q \|_{L^{sq(\cdot)}_{\omega_{\gamma/s}}} \| \chi_Q \|_{L^{(sq(\cdot))'}_{\omega_{-\gamma/s}}} \approx |Q|.
\] 
Since $\omega^{1/s}_{\pm \gamma}= \omega_{\pm \gamma/s}$, by \cite[Theorem 1.5]{Uribe1} we have that the maximal operator $M$ is bounded 
on $L^{(sq(\cdot))'}_{\omega_{\gamma}^{-1/s}}$. By the left-openness property of the set $\mathcal{A}_{(sq(\cdot))'}$ 
(see definition \ref{def Ap var} and \cite[Theorem 3.3]{Ho}), it follows that there exists a constant $\kappa > 1$ such that the maximal operator $M$ is bounded on $L^{(sq(\cdot))'/\kappa}_{\omega_{\gamma}^{-\kappa/s}}$. So, $\omega_{\gamma} \in \mathcal{W}_{q(\cdot)}$ and 
$\| \chi_Q \|_{L^{q(\cdot)}_{\omega_\gamma}} \approx |Q|^{-\alpha/n} \| \chi_Q \|_{L^{p(\cdot)}_{\omega_{\gamma}}}$ for every cube 
$Q \subset \mathbb{R}^{n}$.

\bigskip
\address{
Departamento de Matem\'atica \\
Universidad Nacional del Sur \\
Bah\'{\i}a Blanca, 8000 Buenos Aires \\
Argentina}
{pablo.rocha@uns.edu.ar}

\end{document}